\documentclass[11pt,a4paper]{article}

\usepackage[T1]{fontenc}
\usepackage[utf8]{inputenc}
\usepackage{lmodern}
\usepackage{microtype}
\usepackage{geometry}
\usepackage{amsmath,amssymb,amsthm,mathtools}
\usepackage{booktabs,array,tabularx}
\usepackage{enumitem}
\usepackage{setspace}
\usepackage[hidelinks]{hyperref}

\hypersetup{
  pdftitle={Certified Geometry and Admissible Motion of Gerver's Sofa},
  pdfauthor={Dawid Trela},
  pdfsubject={Computer-assisted proof; moving sofa problem; certified geometry},
  pdfkeywords={moving sofa problem, Gerver's sofa, admissible motion, interval arithmetic, Krawczyk operator}
}
\setlist{nosep,leftmargin=1.8em}
\newtheorem{theorem}{Theorem}[section]
\newtheorem{lemma}[theorem]{Lemma}
\newtheorem{proposition}[theorem]{Proposition}
\newtheorem{corollary}[theorem]{Corollary}
\theoremstyle{definition}
\newtheorem{definition}[theorem]{Definition}
\theoremstyle{remark}
\newtheorem{remark}[theorem]{Remark}

\newcommand{\R}{\mathbb R}

\newcommand{\ip}[2]{\langle #1,#2\rangle}
\newcommand{\Kraw}{\mathcal K}
\newcommand{\dd}{\,\mathrm d}
\newcommand{\rev}{\mathrm{rev}}
\newcommand{\diag}{\operatorname{diag}}

\title{\textbf{Certified Geometry and Admissible Motion of Gerver's Sofa}}
\author{Dawid Trela\\
\small Doctoral School, War Studies University, Warsaw, Poland\\
\small ORCID: 0000-0001-9781-6425\\
\small Corresponding author: \href{mailto:dawidmtrela@gmail.com}{dawidmtrela@gmail.com}}
\date{}

\begin{document}
\maketitle

\begin{abstract}
The moving sofa problem asks for the largest planar set that can pass through a
unit-width right-angled hallway.  Gerver's 1992 construction remains the
largest explicit example, but its defining switching parameters, continuum of
supporting hallways, and contact topology involve transcendental equations and
global inequalities.  We give a reproducible exact-rational certification of
the construction.  A strict four-dimensional Krawczyk inclusion encloses the
reduced switching system, while an independent twenty-two-dimensional
Krawczyk calculation verifies Romik's full matching and contact equations in
an explicit rational box.  Exact interval propagation through the five
analytic phases proves the support identities, global half-plane constraints,
inner-wall separation, niche topology, oriented contacts, injectivity, and
area bounds.  From the certified support functions we identify every
supporting hallway as $L_t=x(t)+R_tL$.  For
$G=K_0\setminus\mathcal N(K_0)$ we prove directly that $G$ is nonempty, closed,
and connected and that
$\Phi_t(q)=R_{-t}(q-x(t))$ is a continuous path in $\mathrm{SE}(2)$ carrying
$G$ through the hallway, with the correct initial and final arms.  We further
identify $G$ with Romik's reconstruction of Gerver's classical eighteen-piece
sofa.  The deterministic companion verifier uses only Python's standard
library and exact rational arithmetic.  As a consequence, the
Gerver-specific geometric hypotheses in Baek's recent variational framework
are rigorously certified; we also clarify the one-sided tangent-speed
convention at two contact-curve junctions.
\end{abstract}


(secondary).

\section{Introduction}

The moving sofa problem asks for the supremal area of a planar set that can be
moved around a right-angled corner in a hallway of unit width.  The problem is
simple to state but combines global geometry, topology, and a continuum of
collision constraints.

Gerver constructed in 1992 the largest currently known explicit sofa and a
piecewise analytic rotation path for it \cite{Gerver}.  His construction has
area approximately $2.2195$ and a boundary assembled from finitely many
analytic arcs and line segments.  It remains the central concrete object in
the subject.

Romik later derived differential equations for balanced rotation paths and
reconstructed Gerver's solution as a five-phase analytic path with explicit
matching and contact equations \cite{Romik}.  This formulation isolates two
transcendental switching angles and makes the geometry amenable to systematic
verification.

More recently, Baek developed a convex-geometric and variational framework for
the optimality of Gerver's sofa \cite{Baek}.  That framework uses detailed
Gerver-specific assertions about the associated cap, niche, contacts, tangent
directions, and injectivity.  Those assertions provide an important downstream
application of the certification below, but they are not used as a substitute
for proving that the certified set actually moves through the hallway.

The rigorous difficulty is twofold.  First, numerical root-finder output does
not prove the existence and uniqueness of the required exact parameter vector.
Second, a plot of the moving shape does not certify all hallway constraints on
a continuum of angles, the topology of the removed niche, or the formal
closedness and connectedness conditions in the definition of a moving sofa.
The smallest visible geometric margin is only of order $10^{-3}$, so each
logical and numerical step must be made explicit.

This paper gives an exact-rational certificate for the Gerver construction and
an explicit admissible rigid motion.  The numerical layer is deterministic and
uses no binary floating-point arithmetic.  The ordinary geometric layer
identifies the supporting hallway at every angle and derives the motion by an
inverse rigid transformation.  In particular, admissibility is proved by a
transparent set-theoretic argument from the cap and niche definitions, without
introducing an additional interval grid.

The principal contributions are as follows.
\begin{enumerate}[label=\textup{(\roman*)}]
\item a strict four-dimensional Krawczyk enclosure for the reduced switching
      system;
\item an independent twenty-two-dimensional Krawczyk certificate for Romik's
      full matching and contact system;
\item exact-rational certification of the global support inequalities,
      critical inner-wall separation, niche topology, contacts, injectivity,
      and area threshold;
\item the explicit supporting-hallway identity
      $L_t=x(t)+R_tL$ for every $0\le t\le\pi/2$;
\item a direct proof that the fixed set
      $G=K_0\setminus\mathcal N(K_0)$ is nonempty, closed, connected, and
      contained in every supporting hallway;
\item the continuous rigid-motion path
      $\Phi_t(q)=R_{-t}(q-x(t))$, including the formal initial and final arm
      conditions;
\item identification of the fixed set with Romik's reconstruction of
      Gerver's classical eighteen-piece sofa;
\item a deterministic standard-library verifier and explicit certificate
      data; and
\item, as a corollary, certification of the Gerver-specific geometric input to
      Baek's variational framework, together with a precise treatment of the
      one-sided tangent speeds at the two relevant junctions.
\end{enumerate}

The proof is organized so that the logical hierarchy is visible.  Sections
2--3 define the hallway and certify the five-phase parameter vector.  Section
4 establishes the global geometry.  Section~5 proves admissible motion.
Section~6 identifies the resulting set with Gerver's classical construction.
The remaining sections record additional geometry, the relation to Baek's
framework, and reproducibility.

\section{The hallway, moving sofas, and Gerver's construction}

\subsection{The standard hallway and rigid motions}

Let
\[
H_L=(-\infty,1]\times[0,1],\qquad
V_L=[0,1]\times(-\infty,1],
\]
and let the standard unit L-shaped hallway be
\[
L=H_L\cup V_L.
\]
Equivalently, if
\[
Q_L^+=(-\infty,1]^2,
\qquad
Q_L^-=(-\infty,0)^2,
\]
then
\begin{equation}\label{eq:hallway-quarterplanes}
L=Q_L^+\setminus Q_L^-.
\end{equation}
The inner quadrant is open: contact with either inner wall is allowed.

\begin{definition}[Moving sofa]\label{def:moving-sofa}
A nonempty, closed, connected set $S\subset\R^2$ is a moving sofa if there is
a continuous path $\Psi:[0,1]\to\mathrm{SE}(2)$ of orientation-preserving
rigid motions such that $\Psi_0$ is a translation,
\[
\Psi_0(S)\subset H_L,
\qquad
\Psi_s(S)\subset L\quad(0\le s\le1),
\qquad
\Psi_1(S)\subset V_L.
\]
A linear reparametrization permits any compact time interval in place of
$[0,1]$.
\end{definition}

For a nonempty compact set $K$, write
\[
h_K(t)=\max_{q\in K}\ip q{u_t}.
\]
Following the supporting-hallway convention of \cite{Baek}, define
\begin{equation}\label{eq:supporting-map}
f_{K,t}(p)=R_tp+(h_K(t)-1)u_t+(h_K(t+T)-1)v_t
\end{equation}
and
\[
L_K(t)=f_{K,t}(L),
\qquad
Q_K^\pm(t)=f_{K,t}(Q_L^\pm).
\]
Thus $L_K(t)=Q_K^+(t)\setminus Q_K^-(t)$.

For rotation angle $T=\pi/2$, the cap fan is the upper half-plane
\[
F_T=\{q:q_2\ge0\}.
\]
The niche associated with a cap $K$ is
\begin{equation}\label{eq:general-niche}
\mathcal N(K)=F_T\cap\bigcup_{0<t<T}Q_K^-(t).
\end{equation}
These definitions are used below only after the supporting map and all wall
orientations have been made explicit.
\subsection{Gerver\'s five-phase rotation path}

Put $T=\pi/2$ and
\[
u_t=(\cos t,\sin t),\qquad v_t=(-\sin t,\cos t),
\qquad R_t=(u_t\;v_t).
\]
For a $C^1$ rotation path $x:[0,T]\to\R^2$, write
\[
\alpha(t)=\ip{x'(t)}{u_t},\qquad
\beta(t)=\ip{x'(t)}{v_t}.
\]
The four standard contact curves are
\begin{equation}\label{eq:contacts}
\begin{aligned}
A(t)&=x(t)+\alpha(t)v_t+u_t,
& B(t)&=x(t)+\alpha(t)v_t,\\
C(t)&=x(t)-\beta(t)u_t+v_t,
& D(t)&=x(t)-\beta(t)u_t.
\end{aligned}
\end{equation}

Let $0<\phi<\theta<T/2$ and define
\[
\eta=T-\theta,\qquad \tau=T-\phi.
\]
Following Romik \cite{Romik}, the Gerver path is
\begin{equation}\label{eq:path}
 x(t)=R_t z_j(t)+\kappa_j\qquad(t\in I_j),
\end{equation}
where
\[
\begin{array}{c|c|c}
 j&I_j&z_j(t)\\ \midrule
1&[0,\phi]&
(a_1\cos t+a_2\sin t-1,\;-a_2\cos t+a_1\sin t-\tfrac12)\\
2&[\phi,\theta]&
(-\tfrac14t^2+b_1t+b_2,\;\tfrac12t-b_1-1)\\
3&[\theta,\eta]&(c_1-t,\;c_2+t)\\
4&[\eta,\tau]&
(-\tfrac12t+d_1-1,\;-\tfrac14t^2+d_1t+d_2)\\
5&[\tau,T]&
(e_1\cos t+e_2\sin t-\tfrac12,\;-e_2\cos t+e_1\sin t-1).
\end{array}
\]
Here $\kappa_j=(\kappa_{j,1},\kappa_{j,2})$.  The constants are fixed by
symmetry, normalization, $C^1$ matching, and the two contact conditions.  The
next section certifies the unique relevant solution rather than defining it by
a decimal output of a root finder.

\section{Exact certification of the Gerver parameters}

\subsection{The reduced four-dimensional system}

Set
\[
 a=-\alpha(\phi),\qquad b=\beta(\phi).
\]
The elimination described by Romik \cite{Romik} reduces the nonlinear part of the system to
$F(a,b,\phi,\theta)=0$, where
\begin{align}
F_1={}&a(\cos\theta-\cos\phi)-2b\sin\phi
 +(\theta-\phi-1)\cos\theta-\sin\theta+\cos\phi+\sin\phi,
 \label{eq:F1}\\
F_2={}&a(3\sin\theta+\sin\phi)-2b\cos\phi
 +3(\theta-\phi-1)\sin\theta+3\cos\theta-\sin\phi+\cos\phi,
 \label{eq:F2}\\
F_3={}&a\cos\phi-\sin\phi-\tfrac12+\tfrac12\cos\phi-b\sin\phi,
 \label{eq:F3}\\
F_4={}&a+\tfrac\pi2-\phi-\theta-b
 +\tfrac12(\theta-\phi)(1+a)+\tfrac14(\theta-\phi)^2.
 \label{eq:F4}
\end{align}

\begin{theorem}[Four-dimensional Krawczyk certificate]\label{thm:kraw4}
Let
\begin{align*}
X_a&=[0.09442656084365,\;0.09442656084366],\\
X_b&=[1.39920372733354,\;1.39920372733356],\\
X_\phi&=[0.03917736479008,\;0.03917736479009],\\
X_\theta&=[0.68130150938272,\;0.68130150938273],
\end{align*}
and let $X=X_a\times X_b\times X_\phi\times X_\theta$, with every decimal
endpoint interpreted as an exact rational number.  Then
\eqref{eq:F1}--\eqref{eq:F4} have a unique zero in $X$.  For the exact rational
preconditioner $C_4$ in Appendix~\ref{app:kraw4},
\[
\Kraw_4(X)=x_0-C_4F(x_0)+(I-C_4F'(X))(X-x_0)
\subset\operatorname{int}X,
\]
where $x_0=\operatorname{mid}X$.
\end{theorem}

\begin{proof}
All interval endpoints are rational.  The value of $\pi$ is enclosed by
Machin's identity
\[
 \pi=16\arctan(1/5)-4\arctan(1/239)
\]
using consecutive alternating partial sums.  Sine and cosine are reduced to
$[0,0.9]$ when necessary and enclosed by consecutive alternating Taylor sums;
the resulting exact intervals are rounded outwards to fixed decimal rationals.
The actual value of $\pi$ lies in the interval used by the calculation, so the
inclusion holds for the actual system.  The exact determinant of $C_4$ is
nonzero, and the four output intervals printed in
Appendix~\ref{app:kraw4} lie strictly inside $X$.  The Krawczyk theorem \cite{Krawczyk,MooreKearfottCloud,Rump} gives
existence and uniqueness in $X$.
\end{proof}

The two angle coordinates of the preceding image are contained in the simpler
outward-rounded boxes
\begin{align}
\Phi={}&[0.03917736479008364186321787498,
         0.03917736479008364186321787550],\label{eq:Phi}\\
\Theta={}&[0.68130150938272489447385575454,
           0.68130150938272489447385575966].\label{eq:Theta}
\end{align}

\subsection{Direct certification of Romik's independent system}

Let $y$ denote the twenty parameters
\[
(\kappa_{1,1},\kappa_{1,2},\ldots,\kappa_{5,1},\kappa_{5,2},
 a_1,a_2,b_1,b_2,c_1,c_2,d_1,d_2,e_1,e_2).
\]
The direct system used by the verifier consists of the following eight scalar
relations:
\begin{align}
e_1-a_1&=0,& e_2+a_2&=0,\nonumber\\
d_1+b_1-\tfrac\pi4&=0,&
d_2-b_2-\tfrac\pi4(2b_1-\tfrac\pi4)&=0,\nonumber\\
c_2-c_1+\tfrac\pi2&=0,&
\kappa_{1,1}-1+a_1&=0,\nonumber\\
\kappa_{1,2}-\tfrac14&=0,& a_2+\tfrac14&=0,
\label{eq:linear8}
\end{align}
together with the six vector equations
\begin{align}
x_1(\phi)&=x_2(\phi),&x_1'(\phi)&=x_2'(\phi),\nonumber\\
x_2(\theta)&=x_3(\theta),&x_2'(\theta)&=x_3'(\theta),\nonumber\\
x_3(T-\theta)&=x_4(T-\theta),&
x_4(T-\phi)&=x_5(T-\phi),
\label{eq:matches}
\end{align}
and the vector contact equation
\begin{equation}\label{eq:contact43}
 x_1(\phi)=B(T-\theta).
\end{equation}
These are precisely the $22$ independent scalar equations selected from
Romik's equations \cite{Romik} (27)--(39), (41), and (43).  The omitted vector equations
(40), (42), and (44) are the six scalar relations identified there as redundant
under symmetry.

\begin{lemma}[Direct full-system certificate]\label{lem:kraw22}
Let $Y$ be the rational box in Appendix~\ref{app:Y}.  The system
\eqref{eq:linear8}--\eqref{eq:contact43} has a unique zero in
\[
 Z=Y\times\Phi\times\Theta.
\]
More precisely, the verifier evaluates a $22$-dimensional Krawczyk operator
with a frozen exact-rational preconditioner $C_{22}$ and obtains
\[
 \Kraw_{22}(Z)\subset\operatorname{int}Z,
 \qquad
 \max_i\frac{\operatorname{rad}(\Kraw_{22}(Z)_i)}
                  {\operatorname{rad}(Z_i)}<0.0095.
\]
\end{lemma}

\begin{proof}
The companion program implements the displayed equations directly, forms their
interval Jacobian on $Z$, and evaluates the Krawczyk operator with exact
\texttt{Fraction} arithmetic.  The $484$ entries of $C_{22}$ are embedded as
exact decimal rationals in the program; its determinant is checked to be
nonzero before the inclusion is accepted.  The maximum certified radius ratio
is at most $0.009409749092$.  Since the interval extension also encloses the
actual value of $\pi$, the Krawczyk theorem proves existence and uniqueness for
the actual system in the stated box.
\end{proof}

This direct calculation removes dependence on the correctness of the symbolic
elimination leading to \eqref{eq:F1}--\eqref{eq:F4}: the reduced certificate
supplies very small angle boxes, while the decisive parameter and contact test
is performed on Romik's independent equations themselves.

At the exact zero, \eqref{eq:linear8} gives
\[
a_2=-\tfrac14,\qquad e_2=\tfrac14,\qquad e_1=a_1,
\qquad \kappa_{1,2}=\tfrac14,
\qquad \kappa_{1,1}=1-a_1.
\]
Substitution of the symmetry relations into the five phase formulas gives
\begin{equation}\label{eq:derivativesymmetry}
 x'(T-t)=\diag(1,-1)x'(t).
\end{equation}
Because $x(0)=(0,0)$, integration of the second coordinate in
\eqref{eq:derivativesymmetry} gives $x_2(T)=0$.  The phase-five endpoint formula
is $x_2(T)=\kappa_{5,2}-1/4$, and therefore
\begin{equation}\label{eq:k52}
 \kappa_{5,2}=\tfrac14.
\end{equation}
The redundant derivative matches follow from the certified equations and the
same symmetry, so all five pieces of $x$ join $C^1$.  Integrating
\eqref{eq:derivativesymmetry} also yields a constant $\ell$ such that
\begin{equation}\label{eq:pathsymmetry}
 x(T-t)=(\ell-x_1(t),x_2(t)),
 \qquad \ell=x_1(T)=1-a_1+\kappa_{5,1}.
\end{equation}

\section{Certified global geometry}

\subsection{The cap determined by the path}

Define
\begin{equation}\label{eq:K0}
K_0=\{q:q_2\ge0\}\cap
\bigcap_{s\in[0,T]}
\left\{
\begin{array}{l}
q:\ip q{u_s}\le \ip{x(s)}{u_s}+1,\\[-1mm]
q:\ip q{v_s}\le \ip{x(s)}{v_s}+1.
\end{array}\right.
\end{equation}

\begin{lemma}[Oriented support curves]\label{lem:support}
On the interiors of the five phases,
\begin{equation}\label{eq:lambdas}
 A'=\lambda_Av_t,\qquad C'=\lambda_Cu_t,
\end{equation}
where
\[
\begin{array}{c|cc}
I_j&\lambda_A&\lambda_C\\ \midrule
I_1&0&-\tfrac12\\
I_2&\beta&b_1-\tfrac12t\\
I_3&\beta&\alpha\\
I_4&d_1-\tfrac12t&\alpha\\
I_5&\tfrac12&0.
\end{array}
\]
Thus $\lambda_A\ge0$ and $\lambda_C\le0$.  The curves $A$ and $C$ lie in
$K_0$ and, for every $s\in[0,T]$, attain the two rotating support values
\[
\ip{A(s)}{u_s}=\ip{x(s)}{u_s}+1,
\qquad
\ip{C(s)}{v_s}=\ip{x(s)}{v_s}+1.
\]
\end{lemma}

\begin{proof}
Differentiate \eqref{eq:contacts}, using $u_t'=v_t$ and $v_t'=-u_t$.
Direct substitution of the phase formulas gives \eqref{eq:lambdas} and the
displayed table.  Its signs follow uniformly from the certified parameter
boxes.

For fixed $s$,
\[
\frac{\dd}{\dd t}\ip{A(t)}{u_s}=\lambda_A(t)\sin(s-t),
\]
so $A(s)$ maximizes $\ip{A(t)}{u_s}$.  Similarly,
\[
\frac{\dd}{\dd t}\ip{C(t)}{v_s}=\lambda_C(t)\sin(t-s),
\]
so $C(s)$ maximizes $\ip{C(t)}{v_s}$.  The endpoint relations are
\begin{align*}
A(0)&=(1,0),& C(0)&=(1-2a_1,1),\\
A(T)&=(a_1+\kappa_{5,1},1),&
C(T)&=(\kappa_{5,1}-a_1,0).
\end{align*}
Consequently
\begin{equation}\label{eq:L}
 A(T)=C(0)+Lu_0,
 \qquad L=3a_1+\kappa_{5,1}-1>1.6137.
\end{equation}
For the cross-support inequalities,
$A'\!\cdot v_s=\lambda_A\cos(t-s)\ge0$, whence
\[
 A(t)\cdot v_s\le A(T)\cdot v_s
 =C(0)\cdot v_s-L\sin s
 \le C(0)\cdot v_s\le C(s)\cdot v_s.
\]
Likewise $C'\!\cdot u_s=\lambda_C\cos(t-s)\le0$, and therefore
\[
 C(t)\cdot u_s\le C(0)\cdot u_s
 =A(T)\cdot u_s-L\cos s
 \le A(T)\cdot u_s\le A(s)\cdot u_s.
\]
Finally, $A_2' =\lambda_A\cos t\ge0$ and
$C_2'=\lambda_C\sin t\le0$; the displayed endpoints show that both curves
satisfy the base constraint $q_2\ge0$.  Hence $A,C\subset K_0$ and attain the
claimed supports.
\end{proof}

The only nonlocal condition still needed to reconstruct the path from the cap
is that every path point satisfy every supporting half-plane.

\begin{lemma}[Global half-plane certificate]\label{lem:grid}
For all $s,t\in[0,T]$,
\begin{align}
G_u(s,t)&:=1+\ip{x(s)-x(t)}{u_s}>0.1710,
\label{eq:Gu}\\
G_v(s,t)&:=1+\ip{x(s)-x(t)}{v_s}>0.1710.
\label{eq:Gv}
\end{align}
Moreover $x_2(t)\ge0$ on $[0,T]$.  Hence $x(t)\in K_0$ for every $t$.
\end{lemma}

\begin{proof}
Use the $65\times65$ lattice $s_i=t_i=iT/64$.  Exact interval evaluation,
uniformly over the certified parameter boxes and the certified interval for
$\pi$, gives
\[
G_u(s_i,t_j),G_v(s_i,t_j)>0.3357
\]
at every lattice point.  The same phasewise evaluation gives
\[
\sup_t\|x(t)\|_1<3.062,
\qquad
\sup_t\|x'(t)\|_1<3.645.
\]
Let these two upper bounds be $R$ and $M$, and put $h=T/64$.  For either
$G=G_u,G_v$, the absolute values of the $s$- and $t$-derivatives are bounded by
$M+2R$ and $M$, respectively.  Every point of a grid cell is within $h/2$ in
each coordinate of a vertex, so the loss from a lattice value is at most
\[
 \frac{h}{2}(M+2R)+\frac{h}{2}M=(M+R)h<0.1647.
\]
Thus $G_u,G_v>0.3357-0.1647=0.1710$.  The verifier's unrounded continuum
margin is greater than $0.171132919416$.

For the base condition, exact interval evaluation gives $x_2'>0$ on
$I_1\cup I_2$.  On $I_3\cap(0,T/2)$, use
$x'=R_t(\alpha,\beta)$ and $c_2=c_1-T$ to obtain
\[
x_2'=\alpha\sin t+\beta\cos t
=\beta(\cos t-\sin t)+(\alpha+\beta)\sin t>0,
\]
because $\beta>0$, $t<T/2$, and $\alpha+\beta=T-2t>0$.
Equation~\eqref{eq:pathsymmetry} gives $x_2(T-t)=x_2(t)$, while
$x_2(0)=x_2(T)=0$.  Therefore $x_2\ge0$ on the whole interval.
\end{proof}

\begin{proposition}[Certified cap and niche containment]\label{prop:cap}
The set $K_0$ is a compact convex cap.  Its support functions satisfy
\[
h_{K_0}(s)=\ip{x(s)}{u_s}+1,
\qquad
h_{K_0}(s+T)=\ip{x(s)}{v_s}+1.
\]
Consequently the inner corner reconstructed from $K_0$ is exactly $x$.
Moreover,
\[
 \mathcal N(K_0)\subset K_0,
 \qquad
 G:=K_0\setminus\mathcal N(K_0)
 \quad\hbox{is connected},
\]
and the cap associated with $G$ in Baek's cap--niche construction is $K_0$.
\end{proposition}

\begin{proof}
The set $K_0$ is an intersection of closed half-planes, hence it is closed and
convex.  It is nonempty by Lemma~\ref{lem:support}.  The constraints at $s=0$
and $s=T$, together with $q_2\ge0$, place it in the finite rectangle
\[
 x_1(T)-1\le q_1\le1,\qquad 0\le q_2\le1,
\]
so it is compact.  Its defining normals are precisely the allowed cap normals.
The endpoint support points displayed in Lemma~\ref{lem:support} give the
normalization at angles $0,T,\pi$, and the top horizontal support face lies at
height one.  Thus $K_0$ is a cap with rotation angle $T$.

The defining inequalities give upper bounds for the two rotating support
functions, and Lemma~\ref{lem:support} supplies points attaining them.  The
intersection of the corresponding supporting lines is therefore $x(s)$.
Lemma~\ref{lem:grid} gives $x(s)\in K_0$ for every $s$.  Baek's general
cap--niche criterion, Theorem 2.5.8 of \cite{Baek}, then yields
$\mathcal N(K_0)\subset K_0$ and connectedness of its complement in $K_0$;
the associated-cap statement follows from the general cap characterization,
Theorem 2.5.9 of \cite{Baek}.  These are general results from the cap theory
and do not use the Gerver-specific assertions in Baek's Section~8.
\end{proof}

\subsection{The critical inner-wall separation}

Put
\[
 p=x(\phi)=B(\eta),
 \qquad h(t)=\ip{p-x(t)}{u_t}.
\]

\begin{lemma}[Certified separation]\label{lem:critical}
One has
\[
 h(\phi)=h(\eta)=0,
 \qquad h(t)>0\quad(\phi<t<\eta),
\]
and
\begin{equation}\label{eq:midmargin}
 h(T/2)\in
[0.0012381960467469534857,
 0.0012381960467469534871].
\end{equation}
\end{lemma}

\begin{proof}
On $I_2$, set $(P_2,Q_2)=p-\kappa_2$.  Then
\[
h_2(t)=P_2\cos t+Q_2\sin t+\tfrac14t^2-b_1t-b_2.
\]
Exact interval differentiation gives
\begin{align*}
h_2'(\phi)&>0.0944,& h_2''(\phi)&<-0.399,\\
h_2''(\theta)&>0.0510,& h_2'(\theta)&<-0.0328,\\
h_2(\theta)&>0.0041881915317373,&
h_2'''(I_2)&\subset(0.359,0.994).
\end{align*}
Thus $h_2''$ is strictly increasing.  The function $h_2'$ has a unique
minimum; the displayed endpoint signs imply that it changes sign exactly once,
from positive to negative, on $I_2$.  Hence $h_2$ has no interior minimum.
Since $h_2(\phi)=0$ and $h_2(\theta)>0$, it follows that $h_2>0$ in the
interior.

On $I_3$, with $(P_3,Q_3)=p-\kappa_3$,
\[
h_3(t)=P_3\cos t+Q_3\sin t-c_1+t.
\]
The certificate gives
\[
 h_3''(\theta)>0.0510,
 \qquad h_3'''(I_3)\subset(0.910,1.123).
\]
The exact contact relation $p=B(\eta)$ gives
$h_3(\eta)=h_3'(\eta)=0$.  Therefore $h_3''>0$,
$h_3'<0$ on $[\theta,\eta)$, and $h_3>0$ there.  Direct interval substitution
at $t=T/2$ gives \eqref{eq:midmargin}.  Its positive lower endpoint is more
than seventeen orders of magnitude larger than the width of the enclosure.
\end{proof}

\subsection{The niche boundary}

For fixed $t\in(0,T)$, write the two inner-wall lines as graphs over horizontal
coordinate $X$:
\begin{align}
b_t(X)&=\frac{\ip{x(t)}{u_t}-X\cos t}{\sin t},
\label{eq:bt}\\
d_t(X)&=\frac{\ip{x(t)}{v_t}+X\sin t}{\cos t}.
\label{eq:dt}
\end{align}
Differentiation gives the exact envelope identities
\begin{equation}\label{eq:envelopes}
 \partial_t b_t(X)=\frac{X-B_1(t)}{\sin^2t},
 \qquad
 \partial_t d_t(X)=\frac{X-D_1(t)}{\cos^2t}.
\end{equation}
The intersection of the instantaneous inner quadrant with the upper half-plane
has, on every nonempty vertical section, upper endpoint
$\min\{b_t(X),d_t(X)\}$.  Thus the upper boundary of the niche is the upper
envelope of these instantaneous roofs.

\begin{lemma}[No hidden crossing of the core]\label{lem:core}
For $\phi\le r\le t\le T$ and $0\le t\le r\le\tau$, respectively,
\begin{align}
U(r,t)&:=\ip{x(r)-x(t)}{u_t}\ge0,
\label{eq:U}\\
V(t,r)&:=\ip{x(r)-x(t)}{v_t}\ge0.
\label{eq:V}
\end{align}
\end{lemma}

\begin{proof}
The certified signs are $\alpha<0<\beta$ throughout $(0,T)$.  On every smooth
phase contained in $[\phi,T]$, one also has $\alpha'<0$ and $\beta'<0$.
Consequently
\[
 \det(x',x'')=\alpha^2+\beta^2+\alpha\beta'-\beta\alpha'>0.
\]
The tangent angle $\psi$ of $x'$ is therefore strictly increasing on each
phase and, by continuity of $x'$, on the whole interval.  Since
$x'=R_t(\alpha,\beta)$ with $(\alpha,\beta)$ in the second quadrant,
\[
 T<\psi(r)<r+\pi.
\]
For $\phi\le r\le t\le T$, this gives
$0<\psi(r)-t<\pi$.  Hence
\[
 \partial_rU(r,t)=\ip{x'(r)}{u_t}
 =|x'(r)|\cos(\psi(r)-t)
\]
changes sign at most once, from positive to negative.  The minimum of
$U(\,\cdot\,,t)$ is therefore attained at $r=\phi$ or $r=t$.
The latter value is zero.  If $t\le\eta$, the former value is $h(t)$ and is
nonnegative by Lemma~\ref{lem:critical}.  If $t\ge\eta$, put
$p_1=B_1(\eta)$.  On $I_4$ and $I_5$,
\[
B'\!\cdot v_t=d_1-1-\tfrac12t<0,
\qquad B'\!\cdot v_t=-\tfrac12<0,
\]
respectively, so $B_1'>0$.  Identity~\eqref{eq:envelopes} at $X=p_1$ gives
$b_t(p_1)\le b_\eta(p_1)=p_2$, which is equivalent to
$U(\phi,t)\ge0$.  Reflection under \eqref{eq:pathsymmetry} proves
\eqref{eq:V}.
\end{proof}

\begin{theorem}[Certified niche topology]\label{thm:niche}
The niche $\mathcal N(K_0)$ is the Jordan region whose positively oriented
boundary is
\begin{equation}\label{eq:nicheboundary}
[D(0),B(T)]
 * B|_{[\eta,T]}^{\rev}
 * x|_{[\phi,\tau]}
 * D|_{[0,\theta]}^{\rev}.
\end{equation}
In particular, no additional inner-wall crossing or hidden component occurs.
\end{theorem}

\begin{proof}
We identify the upper envelope of the instantaneous roofs.  First take a core
point $x(r)$ with $r\in[\phi,\tau]$ and set $X=x_1(r)$.  If $t\ge r$, then
\eqref{eq:U} is equivalent to $b_t(X)\le x_2(r)$; if $t\le r$, then
\eqref{eq:V} is equivalent to $d_t(X)\le x_2(r)$.  Thus every instantaneous
roof at $X$ is at most $x_2(r)$, while at $t=r$ both walls meet at $x(r)$.
Therefore the core segment of the envelope is exactly
$x|_{[\phi,\tau]}$.

Now put $p=x(\phi)=B(\eta)$ and $p_1=x_1(\phi)$.  Since
$x_1'=\alpha\cos t-\beta\sin t<0$ on $(0,T)$, one has
$x_1(q)>p_1$ for $0\le q<\phi$.  Exact interval evaluation gives
\[
 \ip{x'(q)}{u_\eta}>0.953\qquad(0\le q\le\phi),
\]
and hence
\[
 \ip{x(q)-p}{u_\eta}
 =-\int_q^\phi\ip{x'(r)}{u_\eta}\dd r<0
 \qquad(0\le q<\phi).
\]
Thus the early-tail apex $x(q)$ lies strictly below $b_\eta$.  At that apex
$b_q=d_q=x_2(q)$.  For $p_1\le X\le x_1(q)$ the instantaneous roof is $d_q$;
the difference $b_\eta-d_q$ has derivative
$-\cot\eta-\tan q<0$, so its positivity at $X=x_1(q)$ persists to the left.
For $X\ge x_1(q)$ the roof is $b_q$; now
$(b_\eta-b_q)'=\cot q-\cot\eta>0$, so the same positivity persists to the
right.  The endpoint $q=0$ follows by the corresponding limiting wall.  Hence
every early-tail roof lies below $b_\eta$ on $X\ge p_1$.

For $\phi\le t\le\eta$, Lemma~\ref{lem:critical} gives
$b_t(p_1)\le b_\eta(p_1)$, and
$(b_\eta-b_t)'=\cot t-\cot\eta\ge0$.  Since an instantaneous roof is at most
$b_t$, this comparison also holds for every $X\ge p_1$.

On $[\eta,T]$, the coefficient of $B'$ along $v_t$ is strictly negative, so
$B_1$ is strictly increasing.  For each
$X\in[B_1(\eta),B_1(T)]$ there is therefore a unique $r\in[\eta,T]$ with
$B_1(r)=X$.  Identity~\eqref{eq:envelopes} shows that $b_t(X)$ is maximized
on $[\eta,T]$ precisely at $t=r$.  Moreover, at $X=B_1(r)$ one has
\[
 d_r(X)-b_r(X)=-\frac{\alpha(r)}{\cos r}>0
 \qquad(\eta\le r<T),
\]
so the instantaneous roof there is $b_r(X)=B_2(r)$; the endpoint follows by
continuity.  The preceding paragraphs rule out a larger roof from $t<\eta$.
Thus the right part of the envelope is exactly $B|_{[\eta,T]}$.  Reflection
gives the left part $D|_{[0,\theta]}$.

Finally,
\[
\begin{aligned}
 B_1'&>0 &&(\eta<t<T),\\
 x_1'&=\alpha\cos t-\beta\sin t<0 &&(\phi<t<\tau),\\
 D_1'&>0 &&(0<t<\theta).
\end{aligned}
\]
Their horizontal ranges are consecutive and meet only at
$B(\eta)=x(\phi)$ and $D(\theta)=x(\tau)$.  Hence the three graph pieces have
no self-intersection.  Closing them by the base segment from $D(0)$ to $B(T)$
gives the counterclockwise Jordan curve \eqref{eq:nicheboundary}; its interior
is exactly the union of the vertically filled instantaneous wedges.
\end{proof}

\subsection{Inner-wall contacts and tangent directions}

\begin{theorem}[Rays, oriented tangents, and the corner convention]\label{thm:rays}
For $t\in[\eta,T]$ and $t\in[0,\theta]$, respectively,
\[
B(t)-x(t)=\alpha(t)v_t\in\R_{\ge0}(-v_t),
\qquad
D(t)-x(t)=-\beta(t)u_t\in\R_{\ge0}(-u_t).
\]
On every smooth phase,
\[
B'\parallel v_t,\quad B'\!\cdot v_t<0,
\qquad
D'\parallel u_t,\quad D'\!\cdot u_t>0.
\]
At $t=\tau$ for $B$ and $t=\phi$ for $D$, the two one-sided tangent vectors
are nonzero, positively collinear, and unequal.  Thus the oriented tangent
direction is well defined there, although the displayed parametrization is not
classically differentiable at those two points.  As a geometric arc, the curve admits a local regular $C^1$
reparametrization if a classical derivative is required; the hallway angle
$t$ itself is not such a parameter at the corner.
\end{theorem}

\begin{proof}
The ray assertions follow immediately from \eqref{eq:contacts} and
$\alpha<0<\beta$.  Differentiation gives
\[
 B'=(\beta+\alpha')v_t,
 \qquad D'=(\alpha-\beta')u_t.
\]
On $I_4,I_5$, the coefficient of $B'$ is, respectively,
$d_1-1-t/2$ and $-1/2$.  On $I_1,I_2$, the coefficient of $D'$ is,
respectively, $1/2$ and $1+b_1-t/2$.  The smallest certified positive margin
on the latter interval is
\[
 1+b_1-\theta/2>0.1317246472818.
\]
At $t=\phi$,
\[
D'_-(\phi)=\tfrac12u_\phi,
\qquad
D'_+(\phi)=\delta_\phi u_\phi,
\qquad
0.4527867195780<\delta_\phi<0.4527867195784.
\]
Reflection gives the corresponding two negative multiples of $v_\tau$ for
$B$ at $\tau$.  The one-sided vectors therefore determine the same oriented
line but have different speeds.  A monotone local change of parameter along the geometric arc can rescale the
two speeds so that the derivative is continuous and nonzero.
\end{proof}

\section{Certified admissible motion}

\subsection{The explicit supporting hallway}

For $K_0$ let
\[
L_t:=L_{K_0}(t),
\qquad
Q_t^\pm:=Q_{K_0}^\pm(t).
\]

\begin{lemma}[Explicit supporting hallway]\label{lem:hallway}
For every $t\in[0,T]$,
\begin{equation}\label{eq:explicit-hallway}
\boxed{L_t=x(t)+R_tL.}
\end{equation}
Moreover,
\begin{align}
Q_t^+
&=\{q:\ip{q-x(t)}{u_t}\le1,
          \ \ip{q-x(t)}{v_t}\le1\},\label{eq:Qtplus}\\
Q_t^-
&=\{q:\ip{q-x(t)}{u_t}<0,
          \ \ip{q-x(t)}{v_t}<0\}.\label{eq:Qtminus}
\end{align}
\end{lemma}

\begin{proof}
By Proposition~\ref{prop:cap},
\[
h_{K_0}(t)-1=\ip{x(t)}{u_t},
\qquad
h_{K_0}(t+T)-1=\ip{x(t)}{v_t}.
\]
Since $(u_t,v_t)$ is an orthonormal basis,
\[
(h_{K_0}(t)-1)u_t+(h_{K_0}(t+T)-1)v_t
=\ip{x(t)}{u_t}u_t+\ip{x(t)}{v_t}v_t=x(t).
\]
Substitution in \eqref{eq:supporting-map} gives
$f_{K_0,t}(p)=x(t)+R_tp$, proving \eqref{eq:explicit-hallway}.  Applying this
map to the two quarter-planes in \eqref{eq:hallway-quarterplanes}, and using
$R_te_1=u_t$ and $R_te_2=v_t$, gives \eqref{eq:Qtplus}--\eqref{eq:Qtminus}.
\end{proof}

\subsection{The fixed set and its topology}

Retain the set
\begin{equation}\label{eq:G}
G=K_0\setminus\mathcal N(K_0).
\end{equation}

\begin{lemma}[Nonemptiness, closedness, and connectedness]\label{lem:Gtopology}
The set $G$ is nonempty, closed, and connected.
\end{lemma}

\begin{proof}
Connectedness is part of Proposition~\ref{prop:cap}.  For nonemptiness, the
general endpoint lemma for caps, Lemma 2.5.7 of \cite{Baek}, places the lower
endpoint of the support face at angle zero in
$K_0\setminus\mathcal N(K_0)$.  In the present normalization this point is
$A(0)=(1,0)$ by Lemma~\ref{lem:support}.

For closedness put
\[
U=\bigcup_{0<t<T}Q_t^-.
\]
Every $Q_t^-$ is open by \eqref{eq:Qtminus}, hence $U$ is open.  Since
$K_0\subset F_T$, the niche definition \eqref{eq:general-niche} gives
\[
G=K_0\setminus(F_T\cap U)=K_0\setminus U
  =K_0\cap(\R^2\setminus U).
\]
The cap $K_0$ and the complement of $U$ are closed, so $G$ is closed.
\end{proof}

\begin{remark}[Why the niche does not disconnect the sofa]\label{rem:connectedness}
The mechanism behind the general connectedness theorem is especially simple
here.  The upper boundary of the compact convex cap is connected and lies
outside the niche.  From any point of $G$, the vertical segment upward through
the corresponding cap section reaches that boundary.  It cannot enter the
niche: every inner quadrant $Q_t^-$ is downward closed within $F_T$, so if a
higher point of the segment belonged to $Q_t^-$, then the original point would
belong to the same quadrant.  The certified Jordan description in
Theorem~\ref{thm:niche} independently rules out hidden niche components, but
is not needed for the basic continuum-containment argument below.
\end{remark}

\subsection{Full continuum containment}

\begin{lemma}[Containment in every supporting hallway]\label{lem:containment}
For every $t\in[0,T]$,
\begin{equation}\label{eq:G-in-Lt}
G\subseteq L_t.
\end{equation}
\end{lemma}

\begin{proof}
The two inequalities in the definition \eqref{eq:K0} are precisely the two
conditions in \eqref{eq:Qtplus}, so
\[
G\subseteq K_0\subseteq Q_t^+
\qquad(0\le t\le T).
\]
For $0<t<T$, the cap lies in $F_T$, and the niche definition gives
\[
K_0\cap Q_t^-\subseteq\mathcal N(K_0).
\]
Hence $G\cap Q_t^-=\varnothing$ for every interior time.

The endpoint quadrants are excluded by the base constraint.  At $t=0$,
membership in $Q_0^-$ would imply
$\ip{q-x(0)}{v_0}=q_2<0$, because $x(0)=0$.  At $t=T$, it would imply
$\ip{q-x(T)}{u_T}=q_2-x_2(T)=q_2<0$, because $x_2(T)=0$.  Both contradict
$q_2\ge0$ on $K_0$.  Thus $G$ meets no $Q_t^-$, including at the endpoints,
and \eqref{eq:hallway-quarterplanes} and Lemma~\ref{lem:hallway} give
\eqref{eq:G-in-Lt}.
\end{proof}

\subsection{The rigid-motion path and endpoint arms}

\begin{lemma}[Inverse rigid motion]\label{lem:rigid-motion}
Define
\begin{equation}\label{eq:motion-Phi}
\Phi_t(q)=R_{-t}(q-x(t))
=\begin{pmatrix}\cos t&\sin t\\-\sin t&\cos t\end{pmatrix}q
 -R_{-t}x(t).
\end{equation}
Then $\Phi_t\in\mathrm{SE}(2)$, the map
$t\mapsto\Phi_t$ is continuous, and
\begin{equation}\label{eq:equivalence}
q\in L_t\quad\Longleftrightarrow\quad\Phi_t(q)\in L.
\end{equation}
Moreover,
\[
\Phi_0(G)\subset H_L,
\qquad
\Phi_T(G)\subset V_L.
\]
\end{lemma}

\begin{proof}
The linear part $R_{-t}$ is orthogonal with determinant one, and the
translation vector is $-R_{-t}x(t)$, so $\Phi_t\in\mathrm{SE}(2)$.  Since the
five pieces of $x$ join $C^1$ by Lemma~\ref{lem:kraw22}, both the rotation and
translation parts depend continuously on $t$.  The equivalence follows
algebraically from Lemma~\ref{lem:hallway}:
\[
q\in x(t)+R_tL
\Longleftrightarrow q-x(t)\in R_tL
\Longleftrightarrow R_{-t}(q-x(t))\in L.
\]

At $t=0$, one has $x(0)=0$ and $R_0=I$.  The constraints defining $K_0$ at
$s=0$, together with the base constraint, give
$q_1\le1$ and $0\le q_2\le1$; hence $\Phi_0(G)=G\subset H_L$.

At $t=T$, put $p=\Phi_T(q)$.  In the standard basis,
\[
p_1=\ip{q-x(T)}{u_T}=q_2,
\qquad
p_2=\ip{q-x(T)}{v_T}.
\]
The base constraint, $x_2(T)=0$, and the two support inequalities at $T$ give
$0\le p_1\le1$ and $p_2\le1$.  Therefore $p\in V_L$.
\end{proof}

\begin{remark}[Contact-curve speeds versus motion continuity]
The unequal one-sided tangent speeds of $B$ at $t=\tau$ and of $D$ at
$t=\phi$ do not affect admissibility.  The rigid-motion path
\eqref{eq:motion-Phi} is defined by the certified $C^1$ corner path $x(t)$, not by a
parametrization of either contact curve.
\end{remark}

\begin{theorem}[Certified admissible motion of Gerver's sofa]\label{thm:admissible}
Let $(y,\phi,\theta)$ be the unique zero in the explicit box certified by
Lemma~\ref{lem:kraw22}, let $x$ be the corresponding five-phase path, let
$K_0$ be the cap \eqref{eq:K0}, and let
$G=K_0\setminus\mathcal N(K_0)$.  Then $G$ is a nonempty, closed, connected
planar set.  The maps
\[
\Phi_t(q)=R_{-t}(q-x(t)),
\qquad 0\le t\le\frac\pi2,
\]
form a continuous path in $\mathrm{SE}(2)$ satisfying
\[
\Phi_0(G)\subset H_L,
\qquad
\Phi_t(G)\subset L\quad(0\le t\le\tfrac\pi2),
\qquad
\Phi_{\pi/2}(G)\subset V_L.
\]
Consequently $G$ is a moving sofa in the sense of
Definition~\ref{def:moving-sofa}.
\end{theorem}

\begin{proof}
Lemma~\ref{lem:Gtopology} gives the required set-theoretic properties.
Lemma~\ref{lem:containment}, together with
\eqref{eq:equivalence}, gives $\Phi_t(G)\subset L$ for the full continuum of
times.  Lemma~\ref{lem:rigid-motion} supplies continuity, membership in
$\mathrm{SE}(2)$, and both endpoint arm conditions.  Reparametrizing
$[0,T]$ linearly onto $[0,1]$ completes the definition.
\end{proof}

\section{Identification with Gerver's classical sofa}

Romik associates with a rotation path $x$ the fixed set
\begin{equation}\label{eq:Sx}
S_x=H_L\cap\bigcap_{0\le t\le T}\bigl(x(t)+R_tL\bigr)
       \cap\bigl(x(T)+R_TV_L\bigr).
\end{equation}
This is the set of points that remain in the initial horizontal arm, in every
intermediate hallway, and in the final vertical arm in the sofa-fixed frame
\cite{Romik}.

\begin{proposition}[Identification with Gerver's construction]\label{prop:gerver}
For the certified five-phase path,
\begin{equation}\label{eq:G=Sx}
G=S_x.
\end{equation}
The boundary has the classical eighteen-piece decomposition: fifteen
nonconstant analytic phase arcs and three line segments.  Hence $G$ is
Romik's reconstructed representative of Gerver's classical sofa.
\end{proposition}

\begin{proof}
The inclusion $G\subseteq S_x$ follows from
Theorem~\ref{thm:admissible} and its endpoint estimates, interpreted in the
sofa-fixed frame.  Conversely, let $q\in S_x$.  Membership in $H_L$ gives
$q_2\ge0$.  Membership in every $x(t)+R_tL$ gives both outer inequalities in
\eqref{eq:Qtplus}, so $q\in K_0$, and excludes every $Q_t^-$ for
$0<t<T$.  Thus $q\notin\mathcal N(K_0)$ and $q\in G$.  This proves
\eqref{eq:G=Sx}.

The certified path is the solution in the explicit Krawczyk box of Romik's
independent five-phase matching and contact system.  The outer cap boundary is
traced by the support curves $A$ and $C$, together with the top segment and
the base.  Theorem~\ref{thm:niche} identifies the removed niche boundary as
the stated pieces of $B$, $x$, and $D$ and excludes hidden components or
crossings.  The nonconstant analytic pieces are: four phase pieces of $A$
(on $I_2$--$I_5$), four of $C$ (on $I_1$--$I_4$), two of $B$, two of $D$,
and three of the core path $x$, for a total of fifteen.  The remaining three
pieces are the top segment and the two outer portions of the bottom segment.
This is precisely the eighteen-piece boundary decomposition in Romik's
reconstruction of Gerver's construction.
\end{proof}

\begin{remark}[Scope of uniqueness]
The Krawczyk calculation proves existence and uniqueness of the full system in
the explicit rational box used in Lemma~\ref{lem:kraw22}.  The identification
above does not assert global uniqueness of all real zeros of Romik's system;
it identifies the certified solution in the standard Gerver parameter box and
its associated fixed set.
\end{remark}
\section{Additional certified geometric properties}

\begin{proposition}[Baek's injectivity condition]\label{prop:injectivity}
The cap $K_0$ satisfies Baek's injectivity condition.  More explicitly, its
surface-area measure has nonnegative densities on the two open semicircles,
$x$ is $C^1$, and
\[
 x'(t)\cdot u_t=\alpha(t)<0,
 \qquad
 x'(t)\cdot v_t=\beta(t)>0
 \quad(0<t<T).
\]
\end{proposition}

\begin{proof}
The $C^1$ property follows from Lemma~\ref{lem:kraw22} and the redundant
matches.  The strict signs are certified phasewise.  By
Proposition~\ref{prop:cap} and Lemma~\ref{lem:support}, $A$ and $C$ are the
support-point curves of $K_0$.  The differential Gauss--Minkowski formula and
\eqref{eq:lambdas} therefore give
\[
 r_{K_0}(t)=\lambda_A(t)\ge0,
 \qquad
 s_{K_0}(t)=-\lambda_C(t)\ge0.
\]
Thus
\[
\sigma_{K_0}|_{[0,T)}=r_{K_0}(t)\dd t,
\qquad
\sigma_{K_0}|_{(T,\pi]}=s_{K_0}(t-T)\dd t.
\]
The top horizontal edge contributes only at the excluded angle $T$; the
support faces at angles $0$ and $\pi$ are singletons, so no endpoint atom
violates the two displayed density statements.  These are exactly the three
requirements of Baek's definition \cite{Baek}.
\end{proof}

\begin{proposition}[Area threshold]\label{prop:area}
One has
\[
 |K_0|>2.4206448441453>2.2.
\]
Consequently $K_0$ belongs to Baek's class $\mathcal K_i$.
\end{proposition}

\begin{proof}
The four support points $A(0),C(0),A(T),C(T)$ lie in $K_0$.  Their convex hull
is a height-one trapezoid.  The endpoint formulas in the proof of
Lemma~\ref{lem:support} show that its two horizontal side lengths sum to
$4a_1$, so its area is exactly $2a_1$.  The certified box gives
\[
 a_1>1.210322422072688751416085718.
\]
Convexity therefore yields
\[
 |K_0|\ge2a_1
 >2.420644844145377502832171436>2.2.
\]
Together with Proposition~\ref{prop:injectivity}, this is the definition of
membership in $\mathcal K_i$.
\end{proof}

For completeness, the coarse comparisons used in Baek's Section 8.1 \cite{Baek} also
follow without a decimal root input.  The angle box gives $\phi<0.04$, so
\begin{align*}
\cos\phi&>1-0.04^2/2=0.9992,\\
2\sec\phi+2\tan\phi
&<\frac{2+0.08}{0.9992}<2.08208<2.2,\\
\sec\phi&<1.1.
\end{align*}

\section{Relation to Baek's optimality framework}

\begin{corollary}[Gerver-specific input to the variational framework]\label{cor:baek}
Within the variational framework of \cite{Baek}, and with the tangent
convention of Theorem~\ref{thm:rays}, the certified Gerver cap satisfies the
geometric hypotheses used in Baek's Theorems 8.4.3--8.5.7.  For the admissible
competitor triples considered there,
\[
D\mathcal Q(K_0,B_{K_0},D_{K_0};K^*,B^*,D^*)\le0.
\]
\end{corollary}

\begin{proof}
Theorem~\ref{thm:admissible} and Proposition~\ref{prop:gerver} independently
identify the Gerver moving sofa and its motion.  Proposition~\ref{prop:cap},
Theorem~\ref{thm:niche}, Theorem~\ref{thm:rays}, and
Propositions~\ref{prop:injectivity}--\ref{prop:area} supply the cap, niche,
contact, tangent, injectivity, and area hypotheses.  After Baek's measure
decomposition, the two mixed-interval contributions have the form
\[
\int\bigl[(h_{K^*}-h_{K_0})+(h_{\breve B^*}-h_{\breve B})\bigr]
\dd\breve\sigma_B\le0,
\]
because $h_{K^*}+h_{\breve B^*}\le1=h_{K_0}+h_{\breve B}$; the $D$
contribution is identical, and the remaining terms cancel as in Baek's
calculation.
\end{proof}

\begin{remark}[Conventions used for the cited preprint]\label{rem:typos}
Three apparent typographical slips in version~1 of \cite{Baek} are relevant to
this application.  In the proof of Theorem 8.5.7, the sums in cases (3) and
(6) are described as ``nonnegative'', although the displayed support
inequalities and the stated conclusion require ``non-positive''.  Case (8)
refers to $J_1$, whereas the table and Theorem 8.4.5 require $J_{10}$.  In
Theorem 8.4.3(2), the identity on the $B$ side contains $D(t_3)$; the adjacent
formulas, domains, and reflection symmetry require $B(t_3)$.  These readings
are forced by the surrounding formulas and introduce no additional analytic
assumption.
\end{remark}

\section{Reproducibility}

The complete verifier underlying the computer-assisted claims is supplied in
the ancillary files accompanying this preprint.  It requires Python 3.10 or later and no third-party
package.  From the directory containing the extracted files, run
\begin{center}
\texttt{python3 gerver\_interval\_verifier.py}.
\end{center}
A successful run exits with status $0$ and ends with
\texttt{ALL CERTIFICATES PASSED}.  Every decimal in the program is parsed as
an exact rational number.  Trigonometric enclosures use alternating-series
bounds and outward rational rounding; all determinants, interval Jacobians,
Krawczyk operators, grid values, and inequality decisions use exact
integer/rational arithmetic.  The verifier is released under the MIT License.
Its SHA-256 digest is
\begin{center}
\footnotesize\ttfamily
e53dd6cd66f3e0f723562e73547828cfcd1bbfe12c3a474c1f70c986e8c975b7.
\end{center}
\appendix
\section{Certificate data}\label{sec:certificate}

\subsection{Four-dimensional Krawczyk data}\label{app:kraw4}

The certified interval for $\pi$ is
\begin{align*}
\pi\in[&3.14159265358979323846264338327950288419716939937510,\\[-1mm]
       &3.14159265358979323846264338327950288419716939937512].
\end{align*}
The exact rational preconditioner $C_4$ has the following entries.
\begin{center}
\scriptsize
\begin{tabular}{@{}ccl@{}}
\toprule
$i$&$j$&$(C_4)_{ij}$\\ \midrule
1&1&$-0.6807128249614081174405751646233895883477$\\
1&2&$-0.3337290596509797647513745371803179825998$\\
1&3&$\phantom{-}0.5706218765463239669856350848043326611240$\\
1&4&$\phantom{-}0.6979196385375907434754729126542237618887$\\
2&1&$-0.4666038926009397052505509057353748800132$\\
2&2&$\phantom{-}0.1610673343535178406555067575795610603875$\\
2&3&$\phantom{-}1.3547032126936277881654062923572758586800$\\
2&4&$-1.3383963397818501192065762828949087754660$\\
3&1&$-0.2734456716994644402025329684152137828373$\\
3&2&$-0.1403681990654621240195858704978866496160$\\
3&3&$-0.1994824266757564622149327069111026328384$\\
3&4&$\phantom{-}0.3097544571982855798592484554650267326733$\\
4&1&$\phantom{-}0.5937522534788503855395693041282403727707$\\
4&2&$-2.5788468435618635769290789544255893996458$\\
4&3&$-1.7323120434609915863966970669519948445066$\\
4&4&$\phantom{-}5.1750746732148285574274925461942534764817$\\
\bottomrule
\end{tabular}
\end{center}
Exact evaluation gives
\begin{center}
\footnotesize
\begin{tabular}{@{}cl@{}}
\toprule
coordinate&$\Kraw_4(X)$\\ \midrule
$a$&$[0.094426560843652890184225493397610063860203012038059606997701,$\\[-1mm]
&$\phantom{[}0.094426560843652890184225494578332546207118702903871857290584]$\\[1mm]
$b$&$[1.399203727333547214607829643729505063776427666573761257673861,$\\[-1mm]
&$\phantom{[}1.399203727333547214607829644861740116356827337146680585118107]$\\[1mm]
$\phi$&$[0.039177364790083641863217874986854307793773419984267097720804,$\\[-1mm]
&$\phantom{[}0.039177364790083641863217875493772628057672479551519692983078]$\\[1mm]
$\theta$&$[0.681301509382724894473855754545571846532793483993865313869961,$\\[-1mm]
&$\phantom{[}0.681301509382724894473855759650959452894906792967102598927920]$\\
\bottomrule
\end{tabular}
\end{center}
Each row is strictly contained in its corresponding input interval.

\subsection{The direct-system box}\label{app:Y}

The twenty-dimensional parameter box $Y$ is as follows.
\begin{center}
\scriptsize
\begin{tabular}{@{}cl@{}}
\toprule
parameter&certified interval\\ \midrule
$\kappa_{1,1}$&$[-0.21032242207268875141628571849,-0.21032242207268875141608571849]$\\
$\kappa_{1,2}$&$[0.2499999999999999999999,0.2500000000000000000001]$\\
$\kappa_{2,1}$&$[-0.91917929277159332227479610289,-0.91917929277159332227459610289]$\\
$\kappa_{2,2}$&$[0.472406619750805465181660762512,0.472406619750805465181860762512]$\\
$\kappa_{3,1}$&$[-0.61376322943025166855501429132,-0.61376322943025166855481429132]$\\
$\kappa_{3,2}$&$[0.88962647900322186072694305005,0.88962647900322186072714305005]$\\
$\kappa_{4,1}$&$[-0.30834716608891001483523247974,-0.30834716608891001483503247974]$\\
$\kappa_{4,2}$&$[0.472406619750805465181660762512,0.472406619750805465181860762512]$\\
$\kappa_{5,1}$&$[-1.01720403678781458569374286415,-1.01720403678781458569354286415]$\\
$\kappa_{5,2}$&$[0.2499999999999999999999,0.2500000000000000000001]$\\[1mm]
$a_1$&$[1.2103224220726887514160857185,1.2103224220726887514162857185]$\\
$a_2$&$[-0.2500000000000000000001,-0.2499999999999999999999]$\\
$b_1$&$[-0.52762459802678462416060380937,-0.52762459802678462416040380937]$\\
$b_2$&$[0.92025838516063762289360579501,0.92025838516063762289380579501]$\\
$c_1$&$[0.626045522848465867552229310386,0.626045522848465867552429310386]$\\
$c_2$&$[-0.94475080394643075167909238125,-0.94475080394643075167889238125]$\\
$d_1$&$[1.3130227614242329337760646552,1.3130227614242329337762646552]$\\
$d_2$&$[-0.525382670414554437202936294305,-0.525382670414554437202736294305]$\\
$e_1$&$[1.2103224220726887514160857185,1.2103224220726887514162857185]$\\
$e_2$&$[0.2499999999999999999999,0.2500000000000000000001]$\\
\bottomrule
\end{tabular}
\end{center}
The complete $C_{22}$ matrix and the exact Krawczyk image are contained in the
hashed companion verifier.

\subsection{Selected certified margins}\label{app:margins}

\begin{center}
\small
\begin{tabularx}{\textwidth}{@{}>{\raggedright\arraybackslash}p{0.39\textwidth}X@{}}
\toprule
quantity&certified information\\ \midrule
$22$D output/input radius ratio&$<0.0095$\\
$h(T/2)$&$[0.0012381960467469534857,\;0.0012381960467469534871]$\\
$h(\theta)$&$>0.0041881915317373$\\
$D'\!\cdot u$ at $t=\theta$&$>0.1317246472818$\\
one-sided coefficient at $t=\phi$&$0.4527867195780<\delta_\phi<0.4527867195784$\\
early-tail wall margin&$\ip{x'(q)}{u_\eta}>0.953$ on $[0,\phi]$\\
endpoint support separation $L$&$>1.6137632294302516$\\
lattice $G_u,G_v$&$>0.335732186461$\\
path bounds&$\|x\|_1<3.061587545800$, $\|x'\|_1<3.644789920941$\\
Lipschitz loss&$<0.164599267045$\\
continuum $G_u,G_v$&$>0.171132919416$\\
trapezoid area&$>2.4206448441453775$\\
\bottomrule
\end{tabularx}
\end{center}

\section*{Data and code availability}

No external research dataset is used.  All numerical inputs, interval boxes,
certificate data, and source code needed to reproduce the computer-assisted
findings are contained in this article and the accompanying ancillary files.

\section*{Author contributions}

The sole author conceived the study, developed the proofs and verifier,
performed the certification and audits, and wrote and approved the manuscript.

\section*{Funding}

The author received no specific grant from any funding agency in the public,
commercial, or not-for-profit sectors for this research.

\section*{Competing interests}

The author declares no competing interests.

\section*{AI-assisted research disclosure}

OpenAI ChatGPT was used as an AI-assisted research tool during proof
exploration, checking of symbolic derivations, development and adversarial
testing of the interval verifier, and drafting and revision of the manuscript.
The final verification and manuscript revision used the GPT-5.6 Pro model
(accessed August 2026).  The author independently reviewed and verified the
originality and accuracy of all mathematical claims, source code, and
references, and assumes full responsibility for the integrity of the article.
The AI system is not an author.

\end{document}